\documentclass[12pt, reqno]{amsart}
\usepackage{amsfonts}
\usepackage{bbm}
\usepackage{amscd,amsfonts}
\usepackage{amssymb, eucal, amsfonts, amsmath, xypic, latexsym}
\usepackage{pifont}
\usepackage{mathrsfs,color}
\usepackage{amsthm,indentfirst,bm,fancyhdr,dsfont}
\usepackage{graphicx}
\usepackage[all]{xy}
\usepackage[CJKbookmarks=true]{hyperref}

\usepackage{mathrsfs}
\usepackage{amsmath}
\usepackage{amssymb}
\usepackage{hyperref}

\newtheorem{theorem}{Theorem}[section]

\newtheorem{prop}[theorem]{Proposition}

\theoremstyle{definition}

\newtheorem{remark}[theorem]{Remark}

\numberwithin{equation}{section}

\def\ggg{\mathfrak{g}}

\def\sl{\mathfrak{sl}}
\def\p{\mathfrak{p}}

\def\ggg{\mathfrak{g}}

\def\hhh{\mathfrak{h}}

\def\bbb{\mathfrak{b}}

\def\nnn{\mathfrak{n}}

\def\bbf{\mathbb{F}}

\def\K{\mathbb{K}}

\def\0{\bar{0}}
\def\1{\bar{1}}

\def\GL{\text{GL}}

{\vskip-\lastskip\medskip
  \noindent
  {\em #1.}\enspace
  }%
{\qed\par\medskip
  }
\begin{document}

\title[Appendix: The representations of the Lie superalgebra $\p(3)$ in characteristic 3]
{Appendix: The representations of the Lie superalgebra $\p(3)$ in characteristic 3}
\author{Ye Ren}
\address{School of Mathematics, China University of Mining and Technology, Jiangsu 221116, China}\email{TBH560@cumt.edu.cn}


\begin{abstract} Let $\ggg$ be the Lie superalgebra $\p(3)$ of rank 2 over an algebraically closed field $\K$ of characteristic $p=3$. We classify all irreducible modules of $\ggg$, and give the character formulae for irreducible modules.
\end{abstract}

\maketitle
\setcounter{tocdepth}{1}
\section*{Introduction}

In the 1970s, Kac classified the finite-dimensional simple Lie superalgebras over the field of complex numbers (cf. \cite{K77}). Since then, the representation theory of Lie superalgebras has attracted increasing attention (cf. \cite{Du16, PS18, Sh23, WZ09, WZ11, Z24}, etc.). However, their irreducible modules have not been fully understood. This paper classifies all the irreducible modules of $\p(3)$ in characteristic $p=3$, which is a sequel to \cite{R26}, where all the irreducible modules of $\p(3)$ in characteristic $p > 3$ were classified. The method for classifying the irreducible modules of $\p(3)$ in characteristic $3$ is derived from \cite{Y13, Z09}. It is highly dependent on the representations of the even part of $\p(3)$, that is, the structure of the irreducible modules of the special linear Lie algebra $\sl(3)$. The structure of the irreducible modules of $\sl(3)$ in characteristic $3$ is given in \cite{XW17}, and they are more complex than those of $\sl(3)$ in characteristic $p > 3$.

This article is divided into two parts. The first part gives the general setup for the restricted Lie superalgebra $\p(3)$. The second part gives the multiplicities of simple $\p(3)$-modules in the Kac modules $K_{\chi}(\lambda)$ $(\lambda\in \Lambda_{\chi})$. Up to isomorphism of the Kac modules $K_{\chi}(\lambda)$, there are six different types of $p$-characters. We provide specific calculations and proofs for four $p$-characters. Due to the similarities in their calculations and proofs, we omit the proofs for the other two.

\section{Preliminaries}
\subsection{}
Let $\K$ be an algebraically closed field of characteristic $p=3$. In this paper, it is the base field of the theory.
The standard matrix realization of the Lie superalgebra $\ggg=\p(3)=\ggg_{\0}\oplus\ggg_{\1}$ is given as follows:
\[\p(3)=\left\{X=\begin{pmatrix}
A & B \\
C & D \\
\end{pmatrix}|A=-D^t, B^t=B, C^t=-C,  \text{tr}(A)=\text{tr}(D)=0\right\}
.\]
Here, $A, B, C, D$ are $3\times 3$ matrices. The Lie superalgebra $\p(3)$ is a restricted Lie superalgebra (cf. \cite{WZ09}).
Denote by $U(\ggg)$ the universal enveloping algebra of $\ggg$. There is a $[p]$-map such that $x^p-x^{[p]}$ are in the center of $U(\ggg)$ for all $x\in\ggg_{\0}$. For every irreducible $\ggg$-module $M$, there is a $p$-character $\chi$ such that $$(x^p-x^{[p]})_{|M}=\chi(x)^p\text{id}_{|M}.$$ Let $I_{\chi}$ be the ideal of $U(\ggg)$ generated by $x^p-x^{[p]}-\chi(x)^p, x\in\ggg_{\0}$. We call $U_{\chi}(\ggg)=U(\ggg)/I_{\chi}$ the reduced enveloping algebra of $\ggg$.

Let $E_{i,j}$ be the matrix of order $6\times 6$ where the entry in the $i$-th row and $j$-th column equals $1$, with all other entries zero, $1\leq i,j\leq 6$. Let $\epsilon_i$ be the linear map defined by $\epsilon_i(\sum\limits_{k=1}^3a_kE_{k,k})=a_i, a_i\in \K$. Let
$$H_{\epsilon_i-\epsilon_j}= E_{i,i}-E_{j,j}-E_{3+i,3+i}+E_{3+j,3+j}, 1\leq i<j\leq 3,$$
$$X_{\epsilon_i-\epsilon_j}= E_{i,j}-E_{3+j,3+i}, 1\leq i,j\leq 3, i\neq j,$$
$$X_{-\epsilon_i-\epsilon_j}=E_{3+j,i}-E_{3+i,j}, 1\leq i< j\leq 3,$$
$$X_{\epsilon_i+\epsilon_j}= E_{i,3+j}+E_{j,3+i}, 1\leq i< j\leq 3.$$
$$X_{2\epsilon_i}= E_{i,3+i}, 1\leq i\leq 3.$$
Let $\hhh$ be the Cardan subalgebra of $\ggg$ spanned by $H_{\epsilon_i-\epsilon_{i+1}}(1\leq i\leq 2)$. With respect to $\hhh$, there is a standard triangular decomposition of $\ggg=\nnn^-\oplus\hhh\oplus\nnn$. Let $\ggg_{-1}$ be the subalgebra generated by $X_{-\epsilon_i-\epsilon_j}(1\leq i< j\leq 3)$ and
$\ggg_{+1}$ be the subalgebra generated by $X_{\epsilon_i+\epsilon_j}(1\leq i\leq j\leq 3)$. Let $\nnn^-=\nnn_0^-\oplus\ggg_{-1}$, $\nnn=\nnn_0\oplus\ggg_{+1}$.
So $\ggg_{\0}=\nnn_0^-\oplus\hhh\oplus\nnn_0$, and it is isomorphic to the special linear Lie algebra $\mathfrak{sl}(3)$. Let $\bbb_0=\nnn_0\oplus\hhh$, $\bbb=\nnn\oplus\hhh$.
Denote by $\Phi$ the root system. Let $\Phi^+$ $(\Phi^-)$ be the positive (negative) root system. Let $\Phi_0$ consist of even roots.
Then
$$\ggg=\hhh\oplus\bigoplus\limits_{\alpha\in\Phi}\ggg_{\alpha}$$
is the root space decomposition of $\ggg$.
Denote by $\rho_0$ the half sum of the roots of $\nnn_0$. Denote by $W$ the Weyl group of $\Phi_0$. Let $s_{\alpha}\in W$ be the reflection corresponding to $\alpha\in \Phi_0$. For $\lambda\in\hhh^*$, the dot action $s_{\alpha}\cdot\lambda=s_{\alpha}(\lambda+\rho_0)-\rho_0$.

Since $\hhh$ is commutative, for any $\lambda\in\hhh^*$, there is a one-dimensional $U_{\chi}(\hhh)$-module $\K_{\lambda}$. For any $\chi$, define
$$\Lambda_{\chi}=\{\lambda\in\hhh^*| \lambda(h)^p-\lambda(h^{[p]})=\chi^p(h),\forall h\in\hhh\}.$$
Write $\lambda=(r, s)\in \Lambda_{\chi}$ if $\lambda(H_{\epsilon_1-\epsilon_2})=r, \lambda(H_{\epsilon_2-\epsilon_3})=s$. If $\chi=0$, then $r,s\in\bbf_{p}$.

For $\chi\in \ggg_{\0}^*$, if $\gamma\in Aut_p(\ggg_{\0})$ (the group of automorphisms of $\ggg_{\0}$ as a restricted Lie algebra), then
$U_{\chi}(\ggg)\cong U_{\gamma\cdot\chi}(\ggg)$. Note that $\ggg_{\0}\cong\mathfrak{sl}(3)$, and there is a natural surjective homomorphism $\phi: \mathfrak{gl}(3)^*\rightarrow \ggg_{\0}^*$. So we only need to consider the $\GL(3)$-orbit of $\chi\in\ggg_0^*$.
There is a $\GL(3)$-invariant non-degenerate bilinear form on $\mathfrak{gl}(3)$  induced by trace, so $\mathfrak{gl}(3)\cong\mathfrak{gl}(3)^*$.  For each $\chi\in\ggg_{\0}^*$, there exists $g\in\GL(3)$ with $g\cdot\chi$ corresponds to a Jordan matrix. Then we can assume $\chi(\nnn_0)=0$. The $U_{\chi}(\hhh)$-module $\K_{\lambda}$ naturally becomes the $U_{\chi}(\bbb_0)$-module with trivial $\nnn_0$ action. Define the induced $U_{\chi}(\ggg_0)$-module£º
$$Z_{\chi}^0(\lambda)=\text{Ind}_{\bbb_0}^{\ggg_0}(\K_{\lambda}).$$
According to \cite[10.2]{J98}, $Z_{\chi}^0(\lambda)$ has a unique maximal submodule. Denote by $L_{\chi}^0(\lambda)$ the simple quotient of $Z_{\chi}^0(\lambda)$.
Let $v$ be a maximal vector of $Z_{\chi}^0(\lambda)$. Then
$$\{X_{-\epsilon_1+\epsilon_3}^cX_{-\epsilon_2+\epsilon_3}^bX_{-\epsilon_1+\epsilon_2}^av|0\leq a,b,c\leq 2\}$$
form a basis of $Z_{\chi}^0(\lambda)$. We still use $X_{-\epsilon_1+\epsilon_3}^cX_{-\epsilon_2+\epsilon_3}^bX_{-\epsilon_1+\epsilon_2}^av$ to represent elements of $L_{\chi}^0(\lambda)$.
Let $\ggg_{+1}$ act trivially on the simple $U_{\chi}(\ggg_{\0})$-module $L_{\chi}^0(\lambda)$. Then $L_{\chi}^0(\lambda)$ becomes a $U_{\chi}(\ggg_{\0}\oplus\ggg_{+1})$-module. Define the Kac module£º
$$K_{\chi}(\lambda)=\text{Ind}_{\ggg_{\0}\oplus\ggg_{+1}}^{\ggg}(L_{\chi}^0(\lambda))\cong\wedge(\ggg_{-1})\otimes L_{\chi}^0(\lambda).$$


According to \cite[Proposition 1.4]{R26}, there is a $\chi$ in its $\GL(3)$-orbit such that $K_{\chi}(\lambda)$ has a unique maximal submodule. Denote by $L_{\chi}(\lambda)$ the simple quotient of $K_{\chi}(\lambda)$.
For $\lambda=(r, s)\in\Lambda_{\chi}$, set $$\delta(\lambda)=rs(r+s+1)\in\K.$$
If $\delta(\lambda)\neq 0$, then $\lambda$ is called a typical weight. Otherwise, it is called an atypical weight. We have the following proposition:
\begin{prop}(See \cite[Theorem 5.5]{Z24})\label{delta}
Suppose $\delta(\lambda)\neq 0$. Then $K_{\chi}(\lambda)$ is an irreducible $U_{\chi}(\ggg)$-module.
\end{prop}

\section{The representations of $\p(3)$}
\subsection{}
We will investigate the character formula of the irreducible $\ggg$-module $L_{\chi}(\lambda)$, $\lambda\in\Lambda_{\chi}$. There are $p^2$  possibilities for $\lambda$. We only need to consider the $\GL(3)$-orbit of $\chi\in\ggg_0^*$. The representatives of orbits are given below (see \cite{XW17}):
\begin{itemize}
\item[(1)] $\chi_1$ is regular semisimple, i.e. $\chi_1(H_{\epsilon_1-\epsilon_2} )\chi_1(H_{\epsilon_2-\epsilon_3})\chi_1(H_{\epsilon_1-\epsilon_3})\neq 0$, $\chi_1(\nnn_0\oplus\nnn_0^-)=0$.

\item[(2)] $\chi_2$ is subregular semisimple, i.e. $\chi_2(H_{\epsilon_1-\epsilon_2})=0$, $\chi_2(H_{\epsilon_2-\epsilon_3})\chi_2(H_{\epsilon_1-\epsilon_3})\neq 0$, $\chi_2(\nnn_0\oplus\nnn_0^-)=0$.

\item[(3)] $\chi_3=0$.

\item[(4)] $\chi_4(H_{\epsilon_2-\epsilon_3})\neq 0$, $\chi_4(X_{-\epsilon_1+\epsilon_2})=1$, $\chi_4(\nnn_0)=\chi_4(H_{\epsilon_1-\epsilon_2})=\chi_4(X_{-\epsilon_1+\epsilon_3})=\chi_4(X_{-\epsilon_2+\epsilon_3})=0$.

\item[(5)] $\chi_5$ is subregular nilpotent, i.e. $\chi_5(H_{\epsilon_1-\epsilon_2})=\chi_5(H_{\epsilon_2-\epsilon_3})=0$, $\chi_5(X_{-\epsilon_2+\epsilon_3})=1$, $\chi_5(\nnn_0)=0=\chi_5(X_{-\epsilon_1+\epsilon_2})=\chi_5(X_{-\epsilon_1+\epsilon_3})=0$.

\item[(6)] $\chi_6$ is regular nilpotent, i.e. $\chi_6(X_{-\epsilon_1+\epsilon_2} )=\chi_6(X_{-\epsilon_2+\epsilon_3})=1$, $\chi_6(\hhh\oplus\nnn_0)=0$.
\end{itemize}

The character formula of $K_{\chi}(\lambda)$ is clear. So to investigate the character formula of $L_{\chi}(\lambda)$, we only need to investigate the composition factors of $K_{\chi}(\lambda)$. A way to investigate the composition factors of $K_{\chi}(\lambda)$ is to find all the maximal vectors of $K_{\chi}(\lambda)$.
For any vector $v\in K_{\chi}(\lambda)$, $$v=u_1v_1+u_2v_2+u_3v_3,$$
$v_i\in L_{\chi}^0(\lambda), u_i\in \wedge^i(\ggg_{-1}),i=1,2,3.$
According to \cite[Proposition 2.2]{R26},
the vector $v$ is a maximal vector if and only if $u_1v_1,u_2v_2,u_3v_3$ are maximal vectors.
So, we only need to find the homogeneous maximal vectors.
The following three propositions are obtained by direct computation.
\begin{prop}\label{m1}
Let
$m_1=X_{-\epsilon_1-\epsilon_2}w_1+X_{-\epsilon_1-\epsilon_3}w_2+X_{-\epsilon_2-\epsilon_3}w_3\in K_{\chi}(\lambda)$, $w_1$, $w_2$, $w_3\in L_{\chi}^0(\lambda).$ Then
$m_1$ is a maximal vector of weight $(\mu_1,\mu_2)$ if and only if
\begin{gather}
H_{\epsilon_1-\epsilon_2}w_1=\mu_1w_1, H_{\epsilon_1-\epsilon_2}w_2=(\mu_1+1)w_2, H_{\epsilon_1-\epsilon_2}w_3=(\mu_1-1)w_3;\label{2.1}\\
H_{\epsilon_2-\epsilon_3}w_1=(\mu_2+1)w_1, H_{\epsilon_2-\epsilon_3}w_2=(\mu_2-1)w_2, H_{\epsilon_1-\epsilon_2}w_3=\mu_2w_3;\label{2.2}\\
X_{-\epsilon_1+\epsilon_3}w_2+X_{-\epsilon_2+\epsilon_3}w_3=0;\label{2.3}\\
X_{\epsilon_1-\epsilon_2}w_1=X_{\epsilon_1-\epsilon_2}w_2=-w_2+X_{\epsilon_1-\epsilon_2}w_3=0;\label{2.4}\\
X_{\epsilon_2-\epsilon_3}w_1=-w_1+X_{\epsilon_2-\epsilon_3}w_2=X_{\epsilon_2-\epsilon_3}w_3=0.\label{2.5}
\end{gather}
\end{prop}

\begin{prop}\label{m2}
Let $m_2=X_{-\epsilon_1-\epsilon_3}X_{-\epsilon_1-\epsilon_2}u_1+X_{-\epsilon_2-\epsilon_3}X_{-\epsilon_1-\epsilon_2}u_2 +X_{-\epsilon_1-\epsilon_3}X_{-\epsilon_2-\epsilon_3}$ $u_3\in K_{\chi}(\lambda)$, $u_1, u_2, u_3\in L_{\chi}^0(\lambda).$ Then
$m_2$ is a maximal vector of weight $(\mu_1,\mu_2)$ if and only if
\begin{gather}
H_{\epsilon_1-\epsilon_2}u_1=(\mu_1+1)u_1, H_{\epsilon_1-\epsilon_2}u_2=(\mu_1-1)u_2, H_{\epsilon_1-\epsilon_2}u_3=\mu_1u_3;\label{2.6}\\
H_{\epsilon_2-\epsilon_3}u_1=\mu_2u_1, H_{\epsilon_2-\epsilon_3}u_2=(\mu_2+1)u_2, H_{\epsilon_1-\epsilon_2}u_3=(\mu_2-1)u_3;\label{2.7}\\
X_{-\epsilon_2+\epsilon_3}u_2+X_{-\epsilon_1+\epsilon_3}u_1+u_3=X_{-\epsilon_1+\epsilon_3}u_3=
X_{-\epsilon_2+\epsilon_3}u_3=0;\label{2.8}\\
X_{\epsilon_1-\epsilon_2}u_1=X_{\epsilon_1-\epsilon_2}u_2-u_1=X_{\epsilon_1-\epsilon_2}u_3=0;\label{2.9}\\
X_{\epsilon_2-\epsilon_3}u_1=X_{\epsilon_2-\epsilon_3}u_2=u_2+X_{\epsilon_2-\epsilon_3}u_3=0.\label{2.10}
\end{gather}
\end{prop}
\begin{prop}\label{m3}
Let $m_3=X_{-\epsilon_1-\epsilon_3}X_{-\epsilon_2-\epsilon_3}X_{-\epsilon_1-\epsilon_2}w\in K_{\chi}(\lambda)$, $w\in L_{\chi}^0(\lambda).$ Then
$m_3$ is a maximal vector of weight $(\mu_1,\mu_2)$ if and only if
\begin{gather}
H_{\epsilon_1-\epsilon_2}w=\mu_1w, H_{\epsilon_2-\epsilon_3}w=\mu_2w;\label{2.11}\\
X_{-\epsilon_2+\epsilon_3}w=X_{-\epsilon_1+\epsilon_3}w=X_{\epsilon_1-\epsilon_2}w=X_{\epsilon_2-\epsilon_3}w=0.\label{2.12}
\end{gather}
\end{prop}

\subsection{}
Since the structure of the irreducible $U_{\chi}(\ggg_{\0})$-module was given in \cite{XW17}, we can use Proposition \ref{m1}, \ref{m2} and \ref{m3} to compute the maximal vectors of $K_{\chi}(\lambda)$.
\begin{prop} (See \cite[Proposition 2.7]{RB23})
Suppose $\chi=\chi_1$ is regular semisimple and $\lambda\in\Lambda_{\chi}$. Then $K_{\chi}(\lambda)$ is irreducible.
\end{prop}

Let $\chi=\chi_2$, where $\chi_2$ is subregular semisimple. Suppose $\lambda=(r,s)\in\Lambda_{\chi}$. Then $r\in\bbf_p, s\notin\bbf_p$.
As $\delta(\lambda)=rs(r+s+1)$, by Proposition \ref{delta}, we only need to consider the case when $r=0$.
In this case,
$$L^0_{\chi}(\lambda)\cong U_{\chi}(\mathfrak{u})\otimes \K_{\lambda},$$
where $\mathfrak{u}$ is the subalgebra of $\ggg_{\0}$ generated by $X_{-\epsilon_1+\epsilon_3}$ and $X_{-\epsilon_2+\epsilon_3}$.
\begin{prop}\label{chi2}
Let $\chi=\chi_2$ be subregular semisimple and $\lambda\in\Lambda_{\chi}$. Suppose  $\lambda=(0, s), s\notin\bbf_p$. Then $$[K_{\chi}(\lambda)]=[L_{\chi}(\lambda)]+[L_{\chi}(\lambda-\epsilon_1-\epsilon_2)].$$
\end{prop}
\begin{proof}
Let $m_1, m_2, m_3$ be $\bbb$-maximal vectors given in Propositions \ref{m1}, \ref{m2} and \ref{m3}. Let $v$ be the unique nonzero maximal vector of $L_{\chi}^0(\lambda)$ up to a scalar.
By (\ref{2.4}) and (\ref{2.5}), $X_{\epsilon_1-\epsilon_2}w_1=X_{\epsilon_2-\epsilon_3}w_1=0$. So $w_1$ is $\nnn_0$-maximal. Suppose $w_1\neq 0$. We may assume $w_1=v$. The weight of $w_1$ is $(0, s)$. By (\ref{2.1}) and (\ref{2.2}), the weight of $w_2$ is $(1,s+1)$. Let $$w_2=\sum\limits_{a,b,c} k_{c,b,0}X_{-\epsilon_1+\epsilon_3}^{c}X_{-\epsilon_2+\epsilon_3}^{b}v, k_{c,b,0}\in\K.$$
Comparing the weights of $w_1$ and $w_2$, we have
$$2c+b=1.$$
So
$$w_2=k_{2,0,0}X_{-\epsilon_1+\epsilon_3}^{2}v
+k_{0,1,0}X_{-\epsilon_2+\epsilon_3}v+
k_{1,2,0}X_{-\epsilon_1+\epsilon_3}X^2_{-\epsilon_2+\epsilon_3}v.$$
By (\ref{2.4}),
\begin{eqnarray*}
X_{\epsilon_1-\epsilon_2}w_2=k_{2,0,0}X_{-\epsilon_1+\epsilon_3}X_{-\epsilon_2+\epsilon_3}v=0.
\end{eqnarray*}
So $k_{2,0,0}=0$.
By (\ref{2.5}),
\begin{eqnarray*}
X_{\epsilon_2-\epsilon_3}w_2=k_{0,1,0}sv+(2s-1)k_{1,2,0}X_{-\epsilon_1+\epsilon_3}X_{-\epsilon_2+\epsilon_3}v
=v.
\end{eqnarray*}
So $$w_2=\frac{1}{s}X_{-\epsilon_2+\epsilon_3}v.$$
By (\ref{2.1}) and (\ref{2.2}), the weight of $w_3$ is $(-1, s-1)$. So
$$w_3=t_{1,0,0}X_{-\epsilon_1+\epsilon_3}v+t_{2,1,0}X_{-\epsilon_1+\epsilon_3}^2X_{-\epsilon_2+\epsilon_3}v+t_{0,2,0}X_{-\epsilon_2+\epsilon_3}^2v.$$
By (\ref{2.4}),
\begin{eqnarray*}
X_{\epsilon_1-\epsilon_2}w_3=w_2
=2t_{1,0,0}X_{-\epsilon_2+\epsilon_3}v+t_{2,1,0}X_{-\epsilon_1+\epsilon_3}X_{-\epsilon_2+\epsilon_3}^2v
=\frac{1}{s}X_{-\epsilon_2+\epsilon_3}v.
\end{eqnarray*}
So $t_{1,0,0}=-\frac{1}{s},t_{2,1,0}=0$. By (\ref{2.5}),
$$X_{\epsilon_2-\epsilon_3}w_3=0=(2s-2)t_{0,2,0}X_{-\epsilon_2+\epsilon_3}v.$$
So
$$w_3=-\frac{1}{s}X_{\epsilon_3-\epsilon_1}v.$$
Then it can be checked that $$m_1=X_{-\epsilon_1-\epsilon_2}v+\frac{1}{s}X_{-\epsilon_1-\epsilon_3}X_{-\epsilon_2+\epsilon_3}v-\frac{1}{s}X_{-\epsilon_2-\epsilon_3}X_{-\epsilon_1+\epsilon_3}v$$
satisfies (\ref{2.1})-(\ref{2.5}), and it is a maximal vector of $K_{\chi}(\lambda)$.

Suppose $w_1=0$. By (\ref{2.4}) and (\ref{2.5}), $w_2$ is a nonzero multiple of $v$ if $w_2\neq 0$. We may assume $w_2=v$. By (\ref{2.1}) and (\ref{2.2}), the weight of $w_3$ is $(-2,s+1)$. Then
$$w_3=t_{2,0,0}X_{-\epsilon_1+\epsilon_3}^{2}v
+t_{0,1,0}X_{-\epsilon_2+\epsilon_3}v+
t_{1,2,0}X_{-\epsilon_1+\epsilon_3}X^2_{-\epsilon_2+\epsilon_3}v.$$
As
$$X_{\epsilon_1-\epsilon_2}w_3=t_{2,0,0}X_{-\epsilon_1+\epsilon_3}X_{-\epsilon_2+\epsilon_3}v\neq v.$$
It contradicts (\ref{2.4}). So $w_2=0$.
By (\ref{2.4}) and (\ref{2.5}), $w_3$ is a nonzero multiple of $v$ if $w_3\neq 0$. We may assume $w_3=v$. It contradicts (\ref{2.3}). So $w_1=w_2=w_3=0$ and $m_1=0$.
Applying a similar argument, by (\ref{2.6})-(\ref{2.12}), it can be checked that $m_2=m_3=0$.

So $v$ and $X_{-\epsilon_1-\epsilon_2}v+\frac{1}{s}X_{-\epsilon_1-\epsilon_3}X_{-\epsilon_2+\epsilon_3}v-\frac{1}{s}X_{-\epsilon_2-\epsilon_3}X_{-\epsilon_1+\epsilon_3}v$ are the only two maximal vectors of $K_{\chi}(\lambda)$ up to a scalar.
\end{proof}

Let $\chi=\chi_6$ be regular nilpotent, $s_1=s_{\epsilon_1-\epsilon_2}$ and $s_2=s_{\epsilon_2-\epsilon_3}$ be the simple reflections in $W=S_3$. Then we have the
following four $W$-orbits in $\Lambda_{\chi}=\Lambda_{0}$.

(1) $\lambda_1=(0,0)$, $W\cdot\lambda_1=\{(0,0), (1,1)\}$.

(2) $\lambda_2=(2,2)$, $W\cdot\lambda_2=\{(2,2)\}$.

(3) $\lambda_3=(1,0)$, $W\cdot\lambda_3=\{(0,2), (2,1)\}$.

(4) $\lambda_4=(0,1)$, $W\cdot\lambda_4=\{(1,2), (2,0)\}$.

According to \cite[C.3]{J04}, $L^0_{\chi}(\lambda)\cong L^0_{\chi}(w\cdot\lambda)$, and so $K_{\chi}(\lambda)\cong K_{\chi}(w\cdot\lambda)$ for $w\in W$. So to classify the simple $U_{\chi}(\ggg)$-modules, we just need to calculate the composition factors of $K_{\chi}(\lambda)$, for $\lambda=\lambda_i$, $i = 1,2,3,4$. Note that there is a typical weight in the $W$-orbit of $\lambda$ when $\lambda=\lambda_i,i=2,3,4$. Then we have:

\begin{prop}
Let $\chi=\chi_6$ be regular nilpotent, and let $\lambda=\lambda_i,i=2,3,4$ be as listed above. Then
$$[K_{\chi}(\lambda)]=[L_{\chi}(\lambda)].$$
\end{prop}

Let $\chi=\chi_6$ be regular nilpotent and $\lambda=(0,0)$. Let $v$ be the maximal vector of $Z_{\chi}^0(\lambda)$ with weight $\lambda$, by the proof of \cite[Theorem 1]{XW17},
$$v'=X_{-\epsilon_1+\epsilon_2}v+X_{-\epsilon_2+\epsilon_3}v+X_{-\epsilon_1+\epsilon_3}X_{-\epsilon_2+\epsilon_3}X_{-\epsilon_1+\epsilon_2}v
+X_{-\epsilon_2+\epsilon_3}^2X_{-\epsilon_1+\epsilon_2}^2v$$ is a maximal vector of $Z_{\chi}^0(\lambda)$ and $L_{\chi}^0(\lambda)=\ggg_{\0}\cdot v'$.
By direct computation, we can choose a basis of $L_{\chi}^0(\lambda)$:
$$L_{\chi}^0(\lambda)=Span_{\K}\{X_{-\epsilon_1+\epsilon_3}^cX_{-\epsilon_2+\epsilon_3}^bv|0\leq b,c\leq 2\},$$
and we list an equation which will be used in the proof of the next proposition:
\begin{gather}
X_{\epsilon_3-\epsilon_2}^2X_{\epsilon_2-\epsilon_1}v=v-X_{\epsilon_3-\epsilon_1}X_{\epsilon_3-\epsilon_2}v.\label{2.13}
\end{gather}
\begin{prop}\label{chi6}
Let $\chi=\chi_6$ be regular nilpotent and $\lambda=(0,0)$. Then
$$[K_{\chi}(\lambda)]=[L_{\chi}(\lambda)].$$
\end{prop}
\begin{proof}
Let $m_1, m_2, m_3$ be $\bbb$-maximal vectors given in Propositions \ref{m1}, \ref{m2} and \ref{m3}. Let $v^+$ be the $\bbb_0$ maximal vector of $L_{\chi}^0(\lambda)$ with weight $(0,0)$.
By (\ref{2.4}) and (\ref{2.5}), $X_{\epsilon_1-\epsilon_2}w_1=X_{\epsilon_2-\epsilon_3}w_1=0$. So $w_1$ is $\bbb_0$-maximal. Suppose $w_1\neq 0$. We may assume $w_1=v$ of the weight $(\mu_1,\mu_2)$. Actually, $(\mu_1,\mu_2)$ equals $(0, 0)$ or $(1, 1)$. By (\ref{2.1}) and (\ref{2.2}), the weight of $w_2$ is $(\mu_1+1,\mu_2+1)$. Let
$$w_2=k_{0,1,0}X_{-\epsilon_2+\epsilon_3}v+k_{2,0,0}X_{-\epsilon_1+\epsilon_3}^2v+k_{1,2,0}X_{-\epsilon_1+\epsilon_3}X_{-\epsilon_2+\epsilon_3}^2v.$$
By (\ref{2.4}),
\begin{eqnarray*}
X_{\epsilon_1-\epsilon_2}w_2=k_{2,0,0}X_{-\epsilon_1+\epsilon_3}X_{-\epsilon_2+\epsilon_3}v-k_{1,2,0}v=0.
\end{eqnarray*}
So $k_{2,0,0}=k_{1,2,0}=0$.
By (\ref{2.5}),
\begin{eqnarray*}
X_{\epsilon_2-\epsilon_3}w_2=k_{0,1,0}\mu_2v=w_1=v.
\end{eqnarray*}
So $(\mu_1,\mu_2)=(1,1)$, and $w_2=X_{-\epsilon_2+\epsilon_3}^2v^+.$
The weight of $w_3$ is $(\mu_1-1,\mu_2-1)$. So
$$w_3=k_{1,0,0}X_{-\epsilon_1+\epsilon_3}X_{-\epsilon_2+\epsilon_3}v^++k_{0,2,0}v^++k_{2,1,0}X_{-\epsilon_1+\epsilon_3}^2X_{-\epsilon_2+\epsilon_3}^2v^+.$$
By (\ref{2.3}), 
\begin{eqnarray*}
&&X_{-\epsilon_1+\epsilon_3}w_2+X_{-\epsilon_2+\epsilon_3}w_3\\
&=&X_{-\epsilon_1+\epsilon_3}X_{-\epsilon_2+\epsilon_3}^2v^++k_{1,0,0}X_{-\epsilon_1+\epsilon_3}X_{-\epsilon_2+\epsilon_3}^2v^+
+k_{0,2,0}X_{-\epsilon_2+\epsilon_3}v^++k_{2,1,0}X_{-\epsilon_1+\epsilon_3}^2v\\
&=&0.
\end{eqnarray*}
So $$w_3=-X_{-\epsilon_1+\epsilon_3}X_{-\epsilon_2+\epsilon_3}v^+.$$
As 
\begin{eqnarray*}
  &&X_{\epsilon_2-\epsilon_3}w_3=-X_{\epsilon_2-\epsilon_3}X_{-\epsilon_1+\epsilon_3}X_{-\epsilon_2+\epsilon_3}v^+\\
  &=&-X_{-\epsilon_2+\epsilon_3}X_{\epsilon_2-\epsilon_1}v^++X_{-\epsilon_1+\epsilon_3}v^+
  \overset{(\ref{2.13})}{=}2X_{\epsilon_3-\epsilon_1}v^+-X_{\epsilon_3-\epsilon_2}^2v^+\neq 0.
\end{eqnarray*}
It contradicts (\ref{2.5}). So $w_1=0$.
By (\ref{2.4}) and (\ref{2.5}), we may assume $w_2=v$. By (\ref{2.1}) and (\ref{2.2}), the weight of $w_3$ is $(\mu_1+1,\mu_2+1)$. Then
$$w_3=k_{0,1,0}X_{-\epsilon_2+\epsilon_3}v+k_{2,0,0}X_{-\epsilon_1+\epsilon_3}^2v+k_{1,2,0}X_{-\epsilon_1+\epsilon_3}X_{-\epsilon_2+\epsilon_3}^2v.$$
By (\ref{2.3}),
\begin{eqnarray*}
&&X_{-\epsilon_1+\epsilon_3}w_2+X_{-\epsilon_2+\epsilon_3}w_3\\
&=&X_{-\epsilon_1+\epsilon_3}v+k_{0,1,0}X_{-\epsilon_2+\epsilon_3}^2v+k_{2,0,0}X_{-\epsilon_1+\epsilon_3}^2X_{-\epsilon_2+\epsilon_3}v
+k_{1,2,0}X_{-\epsilon_1+\epsilon_3}v\\
&=&0.
\end{eqnarray*}
So $w_3=-X_{-\epsilon_1+\epsilon_3}X_{-\epsilon_2+\epsilon_3}^2v$.

Suppose $v=v^+$. Then 
\begin{eqnarray*}
&&X_{\epsilon_2-\epsilon_3}w_3\\
&=&(\mu_2-2)X_{-\epsilon_1+\epsilon_3}X_{-\epsilon_2+\epsilon_3}v-X_{-\epsilon_2+\epsilon_3}^2X_{-\epsilon_1+\epsilon_2}v\\
&\overset{(\ref{2.13})}{=}&-v+(\mu_2-1)X_{-\epsilon_1+\epsilon_3}X_{-\epsilon_2+\epsilon_3}v\neq 0.
\end{eqnarray*}
Suppose $v=X_{\epsilon_3-\epsilon_2}v^+$. Then
\begin{eqnarray*}
X_{\epsilon_2-\epsilon_3}w_3=-X_{-\epsilon_1+\epsilon_2}v\neq 0.
\end{eqnarray*}
It contradicts (\ref{2.5}), so $$w_2=0.$$
By (\ref{2.4}) and (\ref{2.5}), we may assume $w_3=v$. It contradicts (\ref{2.3}). So $w_1=w_2=w_3=0$ and $$m_1=0.$$
Now consider $m_2$. Suppose $u_1=v$. Then the weight of $u_3$ is $(\mu_1-1,\mu_2-1)$. So
$$u_3=k_{1,0,0}X_{-\epsilon_1+\epsilon_3}v+k_{0,2,0}X_{-\epsilon_2+\epsilon_3}^2v+k_{2,1,0}X_{-\epsilon_1+\epsilon_3}^2X_{-\epsilon_2+\epsilon_3}v.$$
By (\ref{2.8}),
$X_{-\epsilon_2+\epsilon_3}u_3=0=
k_{1,0,0}X_{-\epsilon_1+\epsilon_3}X_{-\epsilon_2+\epsilon_3}v+k_{0,2,0}v+k_{2,1,0}X_{-\epsilon_1+\epsilon_3}^2X_{-\epsilon_2+\epsilon_3}^2v$.
So $u_3=0$. So $u_1=-X_{\epsilon_1-\epsilon_2}X_{\epsilon_2-\epsilon_3}u_3=0,$ contradiction.

Suppose $u_1=0$ and $u_2=v$. Then the weight of $u_3$ is $(\mu_1+1,\mu_2+1)$. Then
$$u_3=k_{0,1,0}X_{-\epsilon_2+\epsilon_3}v+k_{2,0,0}X_{-\epsilon_1+\epsilon_3}^2v+k_{1,2,0}X_{-\epsilon_1+\epsilon_3}X_{-\epsilon_2+\epsilon_3}^2v.$$
By (\ref{2.8}),
$X_{-\epsilon_2+\epsilon_3}u_3=0=k_{0,1,0}X_{-\epsilon_2+\epsilon_3}^2v
+k_{2,0,0}X_{-\epsilon_1+\epsilon_3}^2X_{-\epsilon_2+\epsilon_3}v+k_{1,2,0}X_{-\epsilon_1+\epsilon_3}v$.
So $u_3=0$. Then $u_2=-X_{\epsilon_2-\epsilon_3}u_3=0$, contradiction.

Suppose $u_1=u_2=0$ and $u_3=v$. It contradicts (\ref{2.8}).

Now consider $m_3$. Suppose $w=v$. As $X_{-\epsilon_2+\epsilon_3}w\neq 0$, it contradicts (\ref{2.12}).

\end{proof}

Let $\chi=\chi_5$ be subregular nilpotent, and let $s_2=s_{\epsilon_2-\epsilon_3}$ be the simple reflection in $W_{I}=S_2$. Then we have the
following $W_I$-orbits in $\Lambda_{\chi}=\Lambda_{0}$.

(1) $\lambda_1=(0,0)$, $W_{I}\cdot\lambda_1=\{(0,0), (1,1)\}$.

(2) $\lambda_2=(2,2)$, $W_{I}\cdot\lambda_2=\{(2,2)\}$.

(3) $\lambda_3=(1,0)$, $W_{I}\cdot\lambda_3=\{(1,0), (2,1)\}$.

(4) $\lambda_4=(0,1)$, $W_{I}\cdot\lambda_4=\{(0,1), (2,0)\}$.

(5) $\lambda_5=(1,2)$, $W_{I}\cdot\lambda_5=\{(1,2)\}$.

(6) $\lambda_6=(0,2)$, $W_{I}\cdot\lambda_6=\{(0,2)\}$.

According to \cite[C.3]{J04}, $K_{\chi}(\lambda)\cong K_{\chi}(w\cdot\lambda)$ for $w\in W_{I}$. So to classify the simple $U_{\chi}(\ggg)$-modules, we just need to calculate the composition factors of $K_{\chi}(\lambda)$. Note that there is a typical weight in the $W_{I}$-orbit except for $\lambda=(0,0),(0,1),(0,2)$.

\begin{prop}\label{chi5.1}
Let $\chi=\chi_5$ be subregular nilpotent and $\lambda\in\Lambda_{\chi}$. Suppose $\lambda$ equals $(0,0)$ or $(1,1)$. Then
$$[K_{\chi}(\lambda)]=[L_{\chi}(\lambda)].$$
\end{prop}
\begin{proof}
Let $v^+$ be a $\bbb_0$-maximal vector. By the proof of \cite[Theorem 2]{XW17},
$$L_{\chi}^0(\lambda)=Span_{\K}\{X_{-\epsilon_1+\epsilon_3}^cX_{-\epsilon_2+\epsilon_3}^bv^+|0\leq b,c\leq 2\},$$
and it satisfies
\begin{gather}
X_{-\epsilon_2+\epsilon_3}^2X_{-\epsilon_1+\epsilon_2}v^+=-X_{-\epsilon_1+\epsilon_3}X_{-\epsilon_2+\epsilon_3}v^+.\label{2.14}
\end{gather}
Suppose $m_1, m_2, m_3$, given as in Propositions \ref{m1}, \ref{m2} and \ref{m3}, are $\bbb$-maximal vectors. Let $v$ be a nonzero maximal vector of $L_{\chi}^0(\lambda)$. Then $X_{-\epsilon_2+\epsilon_3}v$ is also a maximal vector.
By (\ref{2.4}) and (\ref{2.5}), $w_1$ is $\nnn_0$-maximal. Suppose $w_1=v$. We may assume the weight of $w_1$ is $(\mu_1,\mu_2)$. Then the weight of $w_2$ is $(\mu_1+1,\mu_2+1)$ and 
$$w_2=k_{0,1,0}X_{-\epsilon_2+\epsilon_3}v+k_{2,0,0}X_{-\epsilon_1+\epsilon_3}^2v+k_{1,2,0}X_{-\epsilon_1+\epsilon_3}X_{-\epsilon_2+\epsilon_3}^2v.$$
By (\ref{2.4}),
\begin{eqnarray*}
X_{\epsilon_1-\epsilon_2}w_2=k_{2,0,0}X_{-\epsilon_1+\epsilon_3}X_{-\epsilon_2+\epsilon_3}v-k_{1,2,0}v=0.
\end{eqnarray*}
So $k_{2,0,0}=k_{1,2,0}=0$.
By (\ref{2.5}),
\begin{eqnarray*}
X_{\epsilon_2-\epsilon_3}w_2=k_{0,1,0}\mu_2v=w_1=v.
\end{eqnarray*}
So $(\mu_1,\mu_2)=(1,1)$ and $w_1=X_{-\epsilon_2+\epsilon_3}v^+$ and $w_2=X_{-\epsilon_2+\epsilon_3}^2v^+$.

The weight of $w_3$ is $(\mu_1-1, \mu_2-1)$. So
$$w_3=t_{1,1,0}X_{-\epsilon_1+\epsilon_3}X_{-\epsilon_2+\epsilon_3}v^++t_{2,2,0}X_{-\epsilon_1+\epsilon_3}^2X_{-\epsilon_2+\epsilon_3}^2v^++t_{0,0,0}v^+.$$
By (\ref{2.3}), $X_{-\epsilon_1+\epsilon_3}w_2+X_{-\epsilon_2+\epsilon_3}w_3=0.$
So $$w_3=-X_{-\epsilon_1+\epsilon_3}X_{-\epsilon_2+\epsilon_3}v^+.$$
As $X_{\epsilon_2-\epsilon_3}w_3=-X_{\epsilon_3-\epsilon_2}X_{\epsilon_2-\epsilon_1}v^++X_{\epsilon_3-\epsilon_1}v^+=2X_{\epsilon_3-\epsilon_1}v^+\neq 0$.
It contradicts (\ref{2.5}).

Suppose $w_1=0$, by (\ref{2.4}) and (\ref{2.5}), $w_2$ is $\bbb_0$-maximal. We may assume $w_2=v$. By (\ref{2.1}) and (\ref{2.2}), the weight of $w_3$ is $(\mu_1+1,\mu_2+1)$. Then
$$w_3=k_{0,1,0}X_{-\epsilon_2+\epsilon_3}v+k_{2,0,0}X_{-\epsilon_1+\epsilon_3}^2v+k_{1,2,0}X_{-\epsilon_1+\epsilon_3}X_{-\epsilon_2+\epsilon_3}^2v.$$
By (\ref{2.3}), $X_{-\epsilon_1+\epsilon_3}w_2+X_{-\epsilon_2+\epsilon_3}w_3=0.$
So, $$w_3=-X_{-\epsilon_1+\epsilon_3}X_{-\epsilon_2+\epsilon_3}^2v.$$
By (\ref{2.5}),
\begin{eqnarray*}
&&X_{\epsilon_2-\epsilon_3}w_3\\
&=&2(2-\mu_2)X_{-\epsilon_1+\epsilon_3}X_{-\epsilon_2+\epsilon_3}v-X_{-\epsilon_2+\epsilon_3}^2X_{-\epsilon_1+\epsilon_2}v\\
&\overset{(\ref{2.14})}{=}&2(1-\mu_2)X_{-\epsilon_1+\epsilon_3}X_{-\epsilon_2+\epsilon_3}v= 0.
\end{eqnarray*}
So $(\mu_1,\mu_2)=(1,1)$ and $w_2=X_{-\epsilon_2+\epsilon_3}v^+$ and $w_3=-X_{-\epsilon_1+\epsilon_3}v^+$.
As $X_{\epsilon_2-\epsilon_3}w_3=-X_{-\epsilon_1+\epsilon_2}v^+\neq 0$, so it contradicts  (\ref{2.5}).
Suppose $w_1=w_2=0$ and $w_3=v$. It contradicts (\ref{2.3}).

Now consider $m_2$. Suppose $u_1=v$. Then the weight of $u_3$ is $(\mu_1-1,\mu_2-1)$. So
$$u_3=k_{1,0,0}X_{-\epsilon_1+\epsilon_3}v+k_{0,2,0}X_{-\epsilon_2+\epsilon_3}^2v+k_{2,1,0}X_{-\epsilon_1+\epsilon_3}^2X_{-\epsilon_2+\epsilon_3}v.$$
By (\ref{2.8}),
$X_{-\epsilon_2+\epsilon_3}u_3=0=
k_{1,0,0}X_{-\epsilon_1+\epsilon_3}X_{-\epsilon_2+\epsilon_3}v+k_{0,2,0}v+k_{2,1,0}X_{-\epsilon_1+\epsilon_3}^2X_{-\epsilon_2+\epsilon_3}^2v$.
So $u_3=0$. So $u_1=-X_{\epsilon_1-\epsilon_2}X_{\epsilon_2-\epsilon_3}u_3=0,$ contradiction.

Suppose $u_1=0$ and $u_2=v$. Then the weight of $u_3$ is $(\mu_1+1,\mu_2+1)$. Then
$$u_3=k_{0,1,0}X_{-\epsilon_2+\epsilon_3}v+k_{2,0,0}X_{-\epsilon_1+\epsilon_3}^2v+k_{1,2,0}X_{-\epsilon_1+\epsilon_3}X_{-\epsilon_2+\epsilon_3}^2v.$$
By (\ref{2.8}),
$X_{-\epsilon_2+\epsilon_3}u_3=0=k_{0,1,0}X_{-\epsilon_2+\epsilon_3}^2v
+k_{2,0,0}X_{-\epsilon_1+\epsilon_3}^2X_{-\epsilon_2+\epsilon_3}v+k_{1,2,0}X_{-\epsilon_1+\epsilon_3}v$.
So $u_3=0$. Then $u_2=-X_{\epsilon_2-\epsilon_3}u_3=0$, contradiction.

Suppose $u_1=u_2=0$ and $u_3=v$. It contradicts (\ref{2.8}).

Now consider $m_3$. Suppose $w=v$. As $X_{-\epsilon_2+\epsilon_3}w\neq 0$, it contradicts (\ref{2.12}).

\end{proof}
\begin{prop}\label{chi5.2}
Let $\chi=\chi_5$ be subregular nilpotent and $\lambda\in\Lambda_{\chi}$. Suppose  $\lambda=(0,1), (2,0)$. Then
$$[K_{\chi}(\lambda)]=[L_{\chi}(\lambda)]+[L_{\chi}(\lambda-\epsilon_1-\epsilon_2)]+[L_{\chi}(\lambda-\epsilon_1-2\epsilon_2+\epsilon_3)].$$
\end{prop}
\begin{proof}
Note that $(0,1)$ and $(2,0)$ are in the same $W_{I}$-orbit. We only need to consider the case $\lambda=(0,1).$
Let $v^+$ be a $\bbb_0$-maximal vector of weight $(0,1)$. By the proof of \cite[Theorem 2]{XW17},
$$L^0_{\chi}(\lambda)\cong U_{\chi}(\mathfrak{u})\otimes \K_{\lambda},$$
where $\mathfrak{u}$ is the subalgebra of $\ggg_{\0}$ generated by $X_{-\epsilon_1+\epsilon_3}$ and $X_{-\epsilon_2+\epsilon_3}$.

Suppose $m_1, m_2, m_3$, given as in Propositions \ref{m1}, \ref{m2} and \ref{m3}, are $\bbb$-maximal vectors. Let $v$ be a nonzero maximal vector of $L_{\chi}^0(\lambda)$. 
By (\ref{2.4}) and (\ref{2.5}), $w_1$ is $\nnn_0$-maximal. Suppose $w_1=v$. We may assume the weight of $w_1$ is $(\mu_1,\mu_2)$. Then the weight of $w_2$ is $(\mu_1+1,\mu_2+1)$ and 
$$w_2=k_{0,1,0}X_{-\epsilon_2+\epsilon_3}v+k_{2,0,0}X_{-\epsilon_1+\epsilon_3}^2v+k_{1,2,0}X_{-\epsilon_1+\epsilon_3}X_{-\epsilon_2+\epsilon_3}^2v.$$
By (\ref{2.4}),
\begin{eqnarray*}
X_{\epsilon_1-\epsilon_2}w_2=k_{2,0,0}X_{-\epsilon_1+\epsilon_3}X_{-\epsilon_2+\epsilon_3}v-k_{1,2,0}v=0.
\end{eqnarray*}
So $k_{2,0,0}=k_{1,2,0}=0$.
By (\ref{2.5}),
\begin{eqnarray*}
X_{\epsilon_2-\epsilon_3}w_2=k_{0,1,0}\mu_2v=w_1=v.
\end{eqnarray*}
So $(\mu_1,\mu_2)=(0,1)$ and $w_2=X_{-\epsilon_2+\epsilon_3}v^+$.

The weight of $w_3$ is $(\mu_1-1, \mu_2-1)$. So
$$w_3=t_{1,0,0}X_{-\epsilon_1+\epsilon_3}v^++t_{2,1,0}X_{-\epsilon_1+\epsilon_3}^2X_{-\epsilon_2+\epsilon_3}v^++t_{0,2,0}X_{-\epsilon_2+\epsilon_3}^2v^+.$$
By (\ref{2.3}), $X_{-\epsilon_1+\epsilon_3}w_2+X_{-\epsilon_2+\epsilon_3}w_3=0.$
So, $$w_3=-X_{-\epsilon_1+\epsilon_3}v^+.$$
Then it can be checked that $$m_1=X_{-\epsilon_1-\epsilon_2}v^++X_{-\epsilon_1-\epsilon_3}X_{-\epsilon_2+\epsilon_3}v^+-X_{-\epsilon_2-\epsilon_3}X_{-\epsilon_1+\epsilon_3}v^+$$
satisfies (\ref{2.1})-(\ref{2.5}), and it is a maximal vector of $K_{\chi}(\lambda)$ of weight $\lambda-\epsilon_1-\epsilon_2$.

Suppose $w_1=0$. By (\ref{2.4}) and (\ref{2.5}), we may assume $w_2=v$. By (\ref{2.1}) and (\ref{2.2}), the weight of $w_3$ is $(\mu_1+1,\mu_2+1)$. Then
$$w_3=k_{0,1,0}X_{-\epsilon_2+\epsilon_3}v+k_{2,0,0}X_{-\epsilon_1+\epsilon_3}^2v+k_{1,2,0}X_{-\epsilon_1+\epsilon_3}X_{-\epsilon_2+\epsilon_3}^2v.$$
By (\ref{2.3}), $X_{-\epsilon_1+\epsilon_3}w_2+X_{-\epsilon_2+\epsilon_3}w_3=0.$
So, $$w_3=-X_{-\epsilon_1+\epsilon_3}X_{-\epsilon_2+\epsilon_3}^2v.$$
Suppose $v=v^+$. 
\begin{eqnarray*}
&&X_{\epsilon_2-\epsilon_3}w_3\\
&=&2(2-\mu_2)X_{-\epsilon_1+\epsilon_3}X_{-\epsilon_2+\epsilon_3}v-X_{-\epsilon_2+\epsilon_3}^2X_{-\epsilon_1+\epsilon_2}v\neq 0.
\end{eqnarray*}
Suppose $v=X_{\epsilon_3-\epsilon_2}^2v^+$. Then 
\begin{eqnarray*}
&&X_{\epsilon_2-\epsilon_3}w_3\\
&=&-X_{\epsilon_2-\epsilon_3}X_{-\epsilon_1+\epsilon_3}X_{-\epsilon_2+\epsilon_3}v=-X_{-\epsilon_2+\epsilon_3}X_{-\epsilon_1+\epsilon_2}v= 0.
\end{eqnarray*}
Then it can be checked that $$m_1=X_{-\epsilon_1-\epsilon_3}X_{-\epsilon_2+\epsilon_3}^2v^+-X_{-\epsilon_2-\epsilon_3}X_{-\epsilon_1+\epsilon_3}X_{-\epsilon_2+\epsilon_3}v^+$$
satisfies (\ref{2.1})-(\ref{2.5}), and it is a maximal vector of $K_{\chi}(\lambda)$ of weight $\lambda-\epsilon_1-2\epsilon_2+\epsilon_3$.

Suppose $w_3=v$. It contradicts (\ref{2.3}).

Now consider $m_2$. Suppose $u_1=v$. Then the weight of $u_3$ is $(\mu_1-1,\mu_2-1)$. So
$$u_3=k_{1,0,0}X_{-\epsilon_1+\epsilon_3}v+k_{0,2,0}X_{-\epsilon_2+\epsilon_3}^2v+k_{2,1,0}X_{-\epsilon_1+\epsilon_3}^2X_{-\epsilon_2+\epsilon_3}v.$$
By (\ref{2.8}),
$X_{-\epsilon_2+\epsilon_3}u_3=0=
k_{1,0,0}X_{-\epsilon_1+\epsilon_3}X_{-\epsilon_2+\epsilon_3}v+k_{0,2,0}v+k_{2,1,0}X_{-\epsilon_1+\epsilon_3}^2X_{-\epsilon_2+\epsilon_3}^2v$.
So $u_3=0$. So $u_1=-X_{\epsilon_1-\epsilon_2}X_{\epsilon_2-\epsilon_3}u_3=0,$ contradiction.

Suppose $u_1=0$ and $u_2=v$. Then the weight of $u_3$ is $(\mu_1+1,\mu_2+1)$. Then
$$u_3=k_{0,1,0}X_{-\epsilon_2+\epsilon_3}v+k_{2,0,0}X_{-\epsilon_1+\epsilon_3}^2v+k_{1,2,0}X_{-\epsilon_1+\epsilon_3}X_{-\epsilon_2+\epsilon_3}^2v.$$
By (\ref{2.8}),
$X_{-\epsilon_2+\epsilon_3}u_3=0=k_{0,1,0}X_{-\epsilon_2+\epsilon_3}^2v
+k_{2,0,0}X_{-\epsilon_1+\epsilon_3}^2X_{-\epsilon_2+\epsilon_3}v+k_{1,2,0}X_{-\epsilon_1+\epsilon_3}v$.
So $u_3=0$. Then $u_2=-X_{\epsilon_2-\epsilon_3}u_3=0$, contradiction.

Suppose $u_1=u_2=0$ and $u_3=v$. It contradicts (\ref{2.8}).

Now consider $m_3$. Suppose $w=v$. As $X_{-\epsilon_2+\epsilon_3}w\neq 0$, it contradicts (\ref{2.12}).

\end{proof}

\begin{prop}\label{chi5.3}
Let $\chi=\chi_5$ be subregular nilpotent and $\lambda\in\Lambda_{\chi}$. Suppose $\lambda=(0,2)$. Then
$$[K_{\chi}(\lambda)]=[L_{\chi}(\lambda)]+[L_{\chi}((\lambda-\epsilon_1-\epsilon_2)]+[L_{\chi}((\lambda-\epsilon_1-\epsilon_3)].$$
\end{prop}

\begin{proof}
Let $v$ be the maximal vector of $L_{\chi}^0(\lambda)$ with weight $(0,2)$. By the proof of \cite[Theorem 2]{XW17},
$$L_{\chi}^0(\lambda)=Span_{\K}\{X_{-\epsilon_1+\epsilon_3}^cX_{-\epsilon_2+\epsilon_3}^bv|0\leq b,c\leq 2\},$$
and it satisfies
\begin{gather}
X_{-\epsilon_2+\epsilon_3}^2X_{-\epsilon_1+\epsilon_2}v=0.\label{2.16}
\end{gather}
Suppose $m_1, m_2, m_3$ are $\bbb$-maximal vectors given in Propositions \ref{m1}, \ref{m2} and \ref{m3} . Let $v$ be a nonzero maximal vector of $L_{\chi}^0(\lambda)$.
By (\ref{2.4}) and (\ref{2.5}), $w_1$ is $\nnn_0$-maximal. Suppose $w_1\neq 0$. We may assume $w_1=v$. Then
$$w_2=k_{0,1,0}X_{-\epsilon_2+\epsilon_3}v+k_{2,0,0}X_{-\epsilon_1+\epsilon_3}^2v+k_{1,2,0}X_{-\epsilon_1+\epsilon_3}X_{-\epsilon_2+\epsilon_3}^2v.$$
By (\ref{2.4}),
\begin{eqnarray*}
X_{\epsilon_1-\epsilon_2}w_2=k_{2,0,0}X_{-\epsilon_1+\epsilon_3}X_{-\epsilon_2+\epsilon_3}v-k_{1,2,0}v=0.
\end{eqnarray*}
So $k_{2,0,0}=k_{1,2,0}=0$.
By (\ref{2.5})
\begin{eqnarray*}
X_{\epsilon_2-\epsilon_3}w_2=2k_{0,1,0}v=v.
\end{eqnarray*}
So $$w_2=-X_{-\epsilon_2+\epsilon_3}v.$$
By the weight of $w_3$, we may assume
$$w_3=t_{1,0,0}X_{-\epsilon_1+\epsilon_3}v+t_{2,1,0}X_{-\epsilon_1+\epsilon_3}^2X_{-\epsilon_2+\epsilon_3}v+t_{0,2,0}X_{-\epsilon_2+\epsilon_3}^2v.$$
By (\ref{2.4}),
\begin{eqnarray*}
X_{\epsilon_1-\epsilon_2}w_3
=-t_{1,0,0}X_{-\epsilon_2+\epsilon_3}v-2t_{2,1,0}X_{-\epsilon_1+\epsilon_3}X_{-\epsilon_2+\epsilon_3}^2v
=-X_{-\epsilon_2+\epsilon_3}v.
\end{eqnarray*}
So $t_{1,0,0}=1, t_{2,1,0}=0.$
By (\ref{2.3}),
\begin{eqnarray*}
&&X_{-\epsilon_1+\epsilon_3}w_2+X_{-\epsilon_2+\epsilon_3}w_3\\
&=&-X_{-\epsilon_1+\epsilon_3}X_{-\epsilon_2+\epsilon_3}v
+X_{\epsilon_3-\epsilon_1}X_{-\epsilon_2+\epsilon_3}v+t_{0,2,0}v
=0.
\end{eqnarray*}
So $k_{0,2,0}=0$, and $$w_3=X_{-\epsilon_1+\epsilon_3}v.$$
It can be checked that $$m_1=X_{-\epsilon_1-\epsilon_2}v-X_{-\epsilon_1-\epsilon_3}X_{-\epsilon_2+\epsilon_3}v+X_{-\epsilon_2-\epsilon_3}X_{-\epsilon_1+\epsilon_3}v$$
satisfies (\ref{2.1})-(\ref{2.5}). So it is a maximal vector of $L_{\chi}(\lambda)$.

Assume $w_1=0$. By (\ref{2.4}) and (\ref{2.5}), $w_2$ is a nonzero multiple of $v$ if $w_2\neq 0$. We may assume $w_2=v$. By (\ref{2.1}) and (\ref{2.2}), the weight of $w_3$ is $(1,1)$. Then
$$w_3=k_{0,1,0}X_{-\epsilon_2+\epsilon_3}v+k_{2,0,0}X_{-\epsilon_1+\epsilon_3}^2v+k_{1,2,0}X_{-\epsilon_1+\epsilon_3}X_{-\epsilon_2+\epsilon_3}^2v.$$
By (\ref{2.4}),
$$X_{\epsilon_1-\epsilon_2}w_3=v=k_{2,0,0}X_{-\epsilon_1+\epsilon_3}X_{-\epsilon_2+\epsilon_3}v-k_{1,2,0}v.$$
So $k_{2,0,0}=0, k_{1,2,0}=-1$.
By (\ref{2.5}),
\begin{eqnarray*}
X_{\epsilon_2-\epsilon_3}w_3
=2k_{0,1,0}v-X_{-\epsilon_2+\epsilon_3}^2X_{-\epsilon_1+\epsilon_2}v
\overset{(\ref{2.16})}{=}2k_{0,1,0}v= 0.
\end{eqnarray*}
So $$w_3=-X_{-\epsilon_1+\epsilon_3}X_{-\epsilon_2+\epsilon_3}^2v.$$
By (\ref{2.3}),
\begin{eqnarray*}
X_{-\epsilon_1+\epsilon_3}w_2+X_{-\epsilon_2+\epsilon_3}w_3
=X_{-\epsilon_1+\epsilon_3}v-X_{-\epsilon_1+\epsilon_3}v=0.
\end{eqnarray*}
It can be checked that $$m_1=X_{-\epsilon_1-\epsilon_3}v-X_{-\epsilon_2-\epsilon_3}X_{-\epsilon_1+\epsilon_3}X_{-\epsilon_2+\epsilon_3}^2v$$
satisfies (\ref{2.1})-(\ref{2.5}). So it is a maximal vector of $L_{\chi}(\lambda)$.

Suppose $w_1=w_2=0$. By (\ref{2.4}) and (\ref{2.5}), $w_3$ is a nonzero multiple of $v$ if $w_3\neq 0$. We may assume $w_3=v$. It contradicts (\ref{2.3}). So $w_1=w_2=w_3=0$ and $$m_1=0.$$

Applying a similar argument, by (\ref{2.6})-(\ref{2.12}), it can be checked that $$m_2=m_3=0.$$

\end{proof}

\begin{prop}
Let $\chi=\chi_4$ and $\lambda\in\Lambda_{\chi}$. Suppose $\lambda$ is atypical. Then
$$[K_{\chi}(\lambda)]=[L_{\chi}(\lambda)].$$
\end{prop}

\begin{proof}
Since $\lambda=(r,s)\in\Lambda_{\chi}$ and $\chi(H_{\epsilon_2-\epsilon_3})=1$, $r\in\bbf_p, s\notin\bbf_p$.
Since $\lambda$ is atypical, $\delta(\lambda)=rs(r+s+1)$, we must have $r=0$. As $\chi(X_{-\epsilon_1+\epsilon_2})=1$, $K_{\chi}(\lambda)\cong K_{\chi}(s_{\epsilon_1-\epsilon_2}\cdot\lambda)$. As $s_{\epsilon_1-\epsilon_2}\cdot(0,s)=(1,s+1)$, $\delta(s_{\epsilon_1-\epsilon_2}\cdot\lambda)\neq 0$.
So $K_{\chi}(\lambda)$ is irreducible.
\end{proof}

Let $\chi=0$ and $\lambda\in\Lambda_{\chi}$ be atypical. We list the basis of irreducible module $L_{\chi}^0(\lambda)$ and some equations required for the computation (cf. \cite[Theorem 3]{XW17}).
\begin{itemize}
\item[(I.1)]  If $\lambda=(0,0)$, then dim$L_{\chi}^0(\lambda)$=1,
$L_{\chi}^0(\lambda)=Span_{\K}\{v\}.$

\item[(I.2)]  If $\lambda=(1,1)$, then dim$L_{\chi}^0(\lambda)$=7,
$$L_{\chi}^0(\lambda)=Span_{\K}\{v,X_1v,X_2X_1v,X_2^2X_1v,X_3X_1v,X_3X_2X_1v,X_2v\},$$
and it satisfies $X_2X_1v=-X_3v$.

\item[(II.1)]  If $\lambda=(1,0)$, then dim$L_{\chi}^0(\lambda)$=3,
 $$L_{\chi}^0(\lambda)=Span_{\K}\{v,X_1v,X_2X_1v\},$$
 and it satisfies $X_2X_1v=X_3v$.

\item[(II.2)]  If $\lambda=(0,2)$, then dim$L_{\chi}^0(\lambda)$=6,
 $$L_{\chi}^0(\lambda)=Span_{\K}\{v,X_3v, X_2v,X_3^2v,X_2^2v,X_3X_2v\}.$$.

\item[(III.1)]  If $\lambda=(0,1)$, then dim$L_{\chi}^0(\lambda)$=3,
 $$L_{\chi}^0(\lambda)=Span_{\K}\{v,X_2v,X_3v\}.$$

\item[(III.2)]  If $\lambda=(2,0)$, then dim$L_{\chi}^0(\lambda)$=6,
 $$L_{\chi}^0(\lambda)=Span_{\K}\{v, X_3v, X_1v, X_3X_1v,X_3X_2X_1v, X_1^2v\},$$
  and it satisfies $X_3v=-X_2X_1v, X_2v=0$.
\end{itemize}

\begin{prop}
Let $\chi=0$ and $\lambda\in\Lambda_{\chi}$. Suppose $\lambda$ is atypical. Then the composition factors of $K_{\chi}(\lambda)$ are given below:
\begin{itemize}
\item[(I.1)]  If $\lambda=(0,0)$, then
$$[K_{\chi}(\lambda)]=[L_{\chi}(\lambda)].$$

\item[(I.2)]  If $\lambda=(1,1)$, then
$$[K_{\chi}(\lambda)]=[L_{\chi}(\lambda)]+[L_{\chi}(\lambda-\epsilon_1-\epsilon_3)].$$

\item[(II.1)] If $\lambda=(1,0)$, then
$$[K_{\chi}(\lambda)]=[L_{\chi}(\lambda)]+[L_{\chi}(\lambda-\epsilon_2-\epsilon_3)].$$

\item[(II.2)]  If $\lambda=(0,2)$, then
$$[K_{\chi}(\lambda)]=[L_{\chi}(\lambda)]+[L_{\chi}(\lambda-\epsilon_1-\epsilon_2)].$$

\item[(III.1)]  If $\lambda=(0,1)$, then
$$[K_{\chi}(\lambda)]=[L_{\chi}(\lambda)]+[L_{\chi}(\lambda-\epsilon_1-\epsilon_2)]+[L_{\chi}(\lambda-\epsilon_1-2\epsilon_2-\epsilon_3)].$$

\item[(III.2)]  If $\lambda=(2,0)$, then
$$[K_{\chi}(\lambda)]=[L_{\chi}(\lambda)]+[L_{\chi}(\lambda-\epsilon_2-\epsilon_3)].$$
\end{itemize}

\end{prop}
\begin{proof}
The argument is parallel to those of Propositions \ref{chi5.1}, \ref{chi5.2} and \ref{chi5.3}. We only list the maximal vectors of $K_{\chi}(\lambda)$ other than $v$.
\begin{itemize}
\item[(I.1)]
\begin{quote}
  $m_1=X_{-\epsilon_2-\epsilon_3}v,$\\
$m_3=X_{-\epsilon_1-\epsilon_3}X_{-\epsilon_2-\epsilon_1}X_{-\epsilon_2-\epsilon_3}v$.
\end{quote}

\item[(I.2)]
\begin{quote}
  $m_1=X_{-\epsilon_1-\epsilon_3}v+X_{-\epsilon_2-\epsilon_3}X_{-\epsilon_1+\epsilon_2}v$.
\end{quote}
\item[(II.1)]
\begin{quote}
  $m_1=X_{-\epsilon_2-\epsilon_3}v$.
\end{quote}
\item[(II.2)]
\begin{quote}
  $m_1=X_{-\epsilon_1-\epsilon_2}v-X_{-\epsilon_1-\epsilon_3}X_{-\epsilon_2+\epsilon_3}v+X_{-\epsilon_2-\epsilon_3}X_{-\epsilon_1+\epsilon_3}v$.
\end{quote}

\item[(III.1)]
\begin{quote}
  $m_1=X_{-\epsilon_1-\epsilon_2}v+X_{-\epsilon_1-\epsilon_3}X_{-\epsilon_2+\epsilon_3}v-X_{-\epsilon_2-\epsilon_3}X_{-\epsilon_1+\epsilon_3}v$,\\
$m_2=X_{-\epsilon_2-\epsilon_3}X_{-\epsilon_1-\epsilon_2}v-X_{-\epsilon_1-\epsilon_3}X_{-\epsilon_2-\epsilon_3}X_{-\epsilon_2+\epsilon_3}v$.
\end{quote}

\item[(III.2)]
\begin{quote}
  $m_1=X_{-\epsilon_2-\epsilon_3}v$.
\end{quote}

\end{itemize}

\end{proof}

\subsection{}
Summarizing the arguments given above, we obtain the main theorem of this article:
\begin{theorem}\label{chara}
With notations as above, the multiplicities of simple modules in $K_{\chi}(\lambda)$ are given as follows:
\begin{itemize}
\item[(1)]  Suppose $\lambda\in\Lambda_{\chi}$ is typical. Then
$$[K_{\chi}(\lambda)]=[L_{\chi}(\lambda)].$$

\item[(2)]  Suppose $\chi=\chi_1,\chi_4$ or $\chi_6$. Then
$$[K_{\chi}(\lambda)]=[L_{\chi}(\lambda)].$$

\item[(3)]  Suppose $\chi=\chi_2$ and $\lambda=(0, s), s\neq 0$. Then
$$[K_{\chi}(\lambda)]=[L_{\chi}(\lambda)]+[L_{\chi}(\lambda-\epsilon_1-\epsilon_2)].$$

\item[(4)]  Suppose $\chi=\chi_3$ and $\lambda\in\Lambda_{\chi}$ is atypical. Then
\[\left\{
    \begin{array}{ll}
      {[K_{\chi}(0)]=[L_{\chi}(0)]+[L_{\chi}(-\epsilon_2-\epsilon_3)]+[L_{\chi}(-2\epsilon_1-2\epsilon_2-2\epsilon_3)]}, & \hbox{$\lambda=(0,0)$;}\\
      {[K_{\chi}(\lambda)]=[L_{\chi}(\lambda)]+[L_{\chi}(\lambda-\epsilon_1-\epsilon_3)]}, & \hbox{$\lambda=(1,1)$;} \\
      {[K_{\chi}(\lambda)]=[L_{\chi}(\lambda)]+[L_{\chi}(\lambda-\epsilon_2-\epsilon_3)]}, & \hbox{$\lambda=(1,0)$;} \\
       {[K_{\chi}(\lambda)]=[L_{\chi}(\lambda)]+[L_{\chi}(\lambda-\epsilon_1-\epsilon_2)]}, & \hbox{$\lambda=(0,2)$;} \\
      {[K_{\chi}(\lambda)]=[L_{\chi}(\lambda)]+[L_{\chi}(\lambda-\epsilon_1-\epsilon_2)]+[L_{\chi}(\lambda-\epsilon_1-2\epsilon_2-\epsilon_3)]}, & \hbox{$\lambda=(0,1)$;} \\
      {[K_{\chi}(\lambda)]=[L_{\chi}(\lambda)]+[L_{\chi}(\lambda-\epsilon_2-\epsilon_3)]}, & \hbox{$\lambda=(2,0)$.} \\
    \end{array}
  \right.
\]

\item[(5)]  Suppose $\chi=\chi_5$ and $\lambda\in\Lambda_{\chi}$ be atypical. Then
\[\left\{
    \begin{array}{ll}
       {[K_{\chi}(\lambda)]=[L_{\chi}(\lambda)]+[L_{\chi}(\lambda-\epsilon_1-\epsilon_2)]+[L_{\chi}(\lambda-\epsilon_1-2\epsilon_2+\epsilon_3)]}, & \hbox{$\lambda=(0,1),(2,0)$;} \\
      {[K_{\chi}(\lambda)]=[L_{\chi}(\lambda)]+[L_{\chi}(\lambda-\epsilon_1-\epsilon_2)]+[L_{\chi}(\lambda-\epsilon_1-\epsilon_3)]}, & \hbox{$\lambda=(0,2)$.}\\
       {[K_{\chi}(\lambda)]=[L_{\chi}(\lambda)]}, & \hbox{others;} \\
    \end{array}
  \right.
\]

\end{itemize}

\end{theorem}

\begin{remark}
The multiplicities of simple modules in $K_{\chi}(\lambda)$ over the field of characteristic $p=3$ given in Theorem \ref{chara} are the same as the multiplicities of simple modules in $K_{\chi}(\lambda)$ over the field of characteristic $p>3$ given in \cite[Theorem 2.11]{R26}
\end{remark}

\end{document}